\documentclass[12pt]{article}
\usepackage{amsmath}
\usepackage{amssymb}
\usepackage{amsthm}
\usepackage{anysize}
\usepackage{pdflscape}
\usepackage{dsfont}
\usepackage{enumerate}
\usepackage{graphicx}
\usepackage{floatrow}

\usepackage[longnamesfirst,authoryear,sort]{natbib}
\usepackage{soul}

\usepackage{multicol}
\usepackage{multirow}
\usepackage[dvips]{epsfig}

\usepackage{float}
\usepackage{color}
\usepackage[latin1]{inputenc}
\usepackage[english]{babel}
\usepackage{xcolor}

\usepackage{bbm}
\usepackage{mathtools}

\numberwithin{equation}{section}

\theoremstyle{plain}
\newtheorem{thm}{Theorem}[section]
\newtheorem{cor}[thm]{Corollary}

\newtheorem{prop}[thm]{Proposition}

\theoremstyle{definition}
\newtheorem{defn}[thm]{Definition}
\theoremstyle{definition}

\newtheorem{oss}{Remark}[section]
\theoremstyle{remark}

\marginsize{25mm}{25mm}{25mm}{25mm}

\newcommand{\I}{\mathds{1}}
\renewcommand{\c}{c_{\gamma, \lambda, \theta}}
\newcommand{\ci}{c_{{\gamma_i}, \lambda_i, \theta_i}}
\newcommand{\cipi}{c_{{\gamma_i}, \pi \lambda_i, \theta_i}}

\newcommand{\E}{\mathrm E}

\begin{document}

\title{Tempered positive Linnik processes and their representations}
\author{Lorenzo Torricelli\footnote{University of Parma, Department of Economics and Management. Email: lorenzo.torricelli@unipr.it} \hspace{2cm} Lucio 
Barabesi\footnote{University of Siena, Department of Economics and Statistics.}  \hspace{2cm}   Andrea Cerioli\footnote{University of Parma, Department of Economics and Management.}  }
\date{\today}

\maketitle

\begin{abstract}
We study several classes of processes associated with the tempered positive Linnik (TPL) distribution, in both the purely absolutely-continuous and mixed law regimes. We explore four main ramifications. Firstly, we analyze several subordinated representations of TPL L\'evy  processes; in particular we establish a stochastic self-similarity  property of positive Linnik (PL)  L\'evy processes, connecting TPL processes with  negative binomial subordination. Secondly, in finite activity regimes we show that  the explicit compound Poisson representations gives rise to innovations following novel  Mittag-Leffler type laws. Thirdly, we characterize two inhomogeneous TPL processes, namely the Ornstein-Uhlenbeck  (OU)  L\'evy-driven processes with stationary distribution and the additive process generated by a TPL law. Finally, we propose a multivariate TPL L\'evy process 
based on a negative binomial mixing methodology of independent interest. Some potential applications of the
considered processes are also outlined in the contexts of statistical
anti-fraud and financial modelling.

\end{abstract}


\noindent {\bf{MSC 2010 classification}}:

\section{Introduction}

In recent years, a large body of literature has been devoted to the tempering of
heavy-tailed laws -- in particular, stable laws -- which prove to be extremely useful in
applications to finance and physics (for a recent introduction to the topic,
see the monograph by \citealt{gra:16}). Indeed, even if heavy-tailed
distributions are well-motivated models in a probabilistic setting,
extremely bold tails are not realistic for most real-world applications. This
drawback has led to the introduction of models which are morphologically
similar to the original distributions even if they display lighter tails. The
initial case for adopting models that are similar to a stable distribution
and with lighter tails is introduced in physics by \cite{man+sta:95} and \cite{kop:95}, and subsequently in economics and finance by the seminal papers of \cite{boy+lev:00} and  \cite{car+al:02}. For recent accounts
on tempered distributions, see e.g. \cite{fal+loe:19} and \cite{gra:19}.

From a static distributional standpoint the  tempered (or ``tilted'') version  $X$ of a random variable (r.v.) $Y$ by a parameter $\theta>0$ is obtained through the Laplace transform 
\begin{equation}\label{eq:introLap}
L_X(s)=\frac{L_Y(\theta+s)}{L_Y(\theta)}\text{ .}
\end{equation}
From this expression 
 it is apparent that the original r.v. 
 and its tempered version have distributions which are practically
indistinguishable for real applications when $\theta$ is small, even if their
tail behaviour is radically different in the sense that the former may have
infinite expectation, while the latter has all the moments finite. 

This methodology is particularly well-adapted to the case in which $Y$ follows an infinitely-divisible distribution, since in that case $X$ is also infinitely-divisible and  expression \eqref{eq:introLap} only involves a simple manipulation of characteristic exponents. Furthermore, the L\'evy measure $\nu_X(dt)$ of $X$ is itself a tilted version of that of $Y$,  meaning that  $\nu_X(dt)=e^{-\theta t}\nu_Y(dt)$. 
The described tempering procedure may therefore be easily embedded in the theory of
L\a'evy processes. The relationship on L\'evy measures highlights that a tempered L\'evy process is one whose small jumps occurrence is essentially indistinguishable from  that of the base process, but whose large jumps are much more rare events.  Additionally, 
tempering retains a natural interpretation in terms of equivalent measure changes in probability spaces. Let a
measure $P_\theta$ be defined by means of the following  martingale density
\begin{equation}
dP_\theta=e^{-\theta Y_t+\phi_Y(\theta)t}dP\text{ ,}
\end{equation}
where $\phi_Y$ is the characteristic (Laplace) exponent of $Y$, which goes under the name of Esscher transform. Under $P_\theta$, the L\'evy process associated to the infinitely-divisible r.v. $Y$  coincides with that associated to $X$. 
This result is of great importance in application fields where the analysis of the process dynamics under equivalent transformation of measures is of relevance, e.g.
option pricing (see \citealt{hub+sga:06}).

In the context of tempering of probability laws, \cite{bar+al:16a} introduce the tempered positive Linnik  (TPL) distribution as a
tilted version of the positive Linnik (PL) distribution considered in \cite{pak:98} and inspired by
the classic paper of \cite{lin:63}. The PL law has received 
increasing interest since it constitutes a generalization of the gamma law
and recovers the positive stable (PS) law as a limiting case (for details,
see \citealt{chr+sch:01}). Hence, its tempered version is suitable
for modelling real data. In addition, the tempering substantially extends the
parameter range of the PL law and accordingly gives rise to two distinct
regimes embedding positive absolutely-continuous distributions, as well as
mixtures of positive absolutely-continuous distributions and the Dirac mass
at zero. 
The latter regime may be useful for modelling
zero-inflated data. Finally, the tempered positive stable (TPS) law (or ``Tweedie distribution''), which 
is central in many recent statistical applications, see e.g. \cite{bar+al:16b}, 
\citealt{fon+al:19}, \citealt{kha+pos:06}, \citealt{ma+has:18},  can also be recovered as a limiting case of the TPL law. A
discrete version of the TPL law is suggested in \cite{bar+al:18}, while
computational issues dealing with the TPL and TPS laws are discussed in \cite{bar:20} and \cite{bar+pra:14,bar+pra:15,bar+pra:19}.

Regarding the theoretical findings on the TPL law, \cite{bar+al:16a}
obtain closed formulas of the probability density function  and the conditional probability density function (under the
two regimes, respectively) of the TPL random variable in terms of the Mittag-Leffler
function and outline the infinite-divisible and self-decomposable character
of the corresponding L\a'evy measures -- as well as their representation as a
mixture of TPS laws with a gamma mixing density.  \cite{kum+al:19a} study
the gamma subordinated representation of the tempered Mittag-Leffler subclass
of TPL L\a'evy process, its moment and covariance properties and provide
alternative derivations of the associated L\a'evy densities and supporting
equations for the probability density function. \cite{leo+al:21} explore in detail
the large deviation theory for TPL processes. \cite{kum+al:19b} instead
analyze Linnik processes -- not necessarily increasing -- and their generalizations.

In this paper, we focus on a number of stochastic processes naturally arising from
the TPL distribution and illustrate their representations and properties. First of all we provide a detailed account of the infinite divisibility property  and unify and clarify the L\'evy-Khinctine structure of a TPL law. We also prove additional properties of TPL laws, such as geometric infinite divisibility. We then study the  subordinated structure of the TPL L\'evy process, showing that besides the defining characterization as a gamma-subordinated law, such laws enjoy numerous representations in terms of a negative binomial subordination. This is in turn connected to the stochastic self-similarity property -- as introduced by \cite{koz+al:06} -- of the PL subordinator with respect to the negative binomial subordinator.  We further make clear the role of the tail parameter $\gamma$ in determining two distinct regimes for the processes associated to the TPL distribution. 
Whenever $\gamma \in (0,1]$, the TPL law is absolutely continuous and 
an infinite activity process occurs. In contrast, when $\gamma<0$,  the TPL has a mixed absolutely-continuous and point mass expression. We then find that the corresponding L\'evy process is a compound Poisson process which we show to feature increments of ``logarithmic'' Mittag-Leffler type. 
 The absolutely-continuous case instead corresponds to a self-decomposable family distributions to which using classic theory (\citealt{bn:97}, \citealt{sat:91}) we are able to associate an Ornstein-Uhlenbeck (OU) L\'evy-driven process with TPL stationary distribution and an additive TPL process. We characterize such processes. In  particular, the L\'evy driving noise of the OU process with TPL stationary distribution is a compound Poisson process with tempered Mittag-Leffler distribution, a probability law which as far as the authors are aware has not been considered before. Finally, we concentrate on the multivariate TPL L\'evy process constructed from a TPL L\'evy process with independent marginals subordinated to a negative binomial subordinator. Such a construction  makes again critical use of the stochastic self-similarity property, and can be easily generalized to different subordinand processes. Furthermore, it encompasses some well-known multivariate distributions of common use in
the statistical environment.

The paper is organized as follows. In Section 2 -- after reviewing some basic known properties -- we discuss the
 infinite divisibility and the L\a'evy-Khintchine representation, as
well as  the self-decomposability and geometric infinite divisibility, of the TPL law.
Section 3 is devoted to TPL L\a'evy subordinators and their various
representations. In Section 4, we consider the OU
L\a'evy-driven processes with stationary TPL distribution and the additive
process generated by a TPL law. In Section 5, we propose the natural
multivariate version of the TPL law, and the connected L\a'evy process. Finally, in Section 6,  applications for statistical anti-fraud are considered.





\section{The tempered positive Linnik distribution}\label{sec:dist}

If $X$ represents a positive random variable (r.v.) on a probability space 
$(\Omega,\mathcal{F}, P)$, we denote its Laplace transform $L_X(s)=\E[e^{-sX}]$, for all values of $s \in \mathbb C$ for which such expectation exists. 

The PL family of laws PL$(\gamma, \lambda, \delta)$ was introduced by \cite{pak:98}, on the basis of the original suggestion of 
\cite{lin:63}. A  PL r.v. $V$ is characterized by the Laplace transform

\begin{equation}\label{eq:PLdef}
L_{V}(s)=\left(\frac{1}{1+ \lambda s ^ \gamma}\right)^{\delta}, \quad
 \text{Re}(s)>0 ,
\end{equation}
$( \gamma,\lambda, \theta) \in \times (0,1]\times \mathbb{R}_+ \times \mathbb{R}_+ $.
For details on Linnik-type laws see e.g. \cite{chr+sch:01}, or more recently \cite{kor+al:20} and references therein.

\cite{bar+al:16a} propose a new family of distributions which is a tempered version of the PL family. By slightly modifying the parametrization thereby proposed,
the   TPL r.v. $X$ is defined as a member of the four-parameter family  TPL$(\gamma, \lambda, \delta, \theta) $ with Laplace transform
given by

\begin{equation}\label{eq:TPLdef}
L_{X}(s)=\left(\frac{1}{1+ \text{sgn}(\gamma)\lambda((\theta+s)^\gamma
-\theta^\gamma)}\right)^{\delta}, \quad
 \text{Re}(s)>0 ,
\end{equation}
with  parameter space for $(\gamma,\lambda,\delta,\theta)$ given by
\begin{equation}
 S=\{ (-\infty,1]\setminus \{ 0\} \times\mathbb{R}_{+}\times%
\mathbb{R}_{+}\times\mathbb{R}_{+}\}\cup\{(0,1]\times\mathbb{R}_{+}\times%
\mathbb{R}_{+}\times\{0\}\}.
\end{equation}

The terminology is motivated by the fact that the genesis of this distribution  for $\gamma \in (0,1)$ is that of tempering PL random variables
in a way analogous to the classic tempering of stable laws. If $V$ is PL$(\gamma, \lambda, \delta)$ and $\theta >0$ then
\begin{equation}\label{eq:tilt}
L_{X}(s)= \frac{L_{V}(\theta+s)}{L_{V}(\theta)}=\left(\frac{1}{1+\lambda'((\theta+s)^\gamma-\theta^\gamma)}\right)^{\delta}
\end{equation}
with $\lambda'= \lambda/(1+\lambda \theta^\gamma)$. See Subsection \ref{sec:sdgid} further on  for more on this analogy.

\medskip

 The TPL family encompasses the PL
law  for $\theta=0$, the  Mittag-Leffler law proposed by
\cite{pil:90} for $\delta=1$ and $\theta=0$,
 and the gamma law for $\gamma=1$, or -- alternatively -- for $\gamma\in (0,1]$ and
$\lambda\theta^\gamma=1$. Furthermore, a TPS$(\gamma,\lambda,\theta)$ can be obtained as a limit in distribution of a TPL$(\gamma, \delta \lambda, \delta, \theta)$ as $\delta \rightarrow \infty$.
In addition from  \eqref{eq:TPLdef} we see that the TPL family is closed under convolution. More precisely  let  $(X_k)_{k \in \mathbb N}$   be a sequence of independent r.v.s with \text{TPL}$(\gamma, \lambda, \delta_k, \theta)$ distribution; then 
$\sum_{k=1}^n X_k$ has  \text{TPL}$\left(\gamma, \lambda, \sum_{k=1}^n \delta_k, \theta\right)$ distribution.

\medskip
The fact that in \eqref{eq:TPLdef} the parameter $\gamma$ is allowed to be negative has many implications and is one of the central aspects of this paper. To begin with, the two parameters subsets  $(-\infty, 0)$ and $(0,1]$ for $\gamma$ determine distinct regimes in the Lebesgue decomposition of the law of a TPL variable. 

Denote with $E_{a, b}^c(z)$, $z  \in \mathbb C$, the \cite{pra:71} three-parameter Mittag-Leffler function  \begin{equation}\label{eq:convolution}
E_{a, b}^c(z)=\sum_{k=0}^\infty \frac{(c)_k z^k}{k!\Gamma(ak+b)}, \qquad \text{Re}(a) >0, \quad \text{Re}(b)>0, \quad c \in \mathbb C,
\end{equation}
where $(c)_k=c (c+1) \ldots (c+k-1)$ is the Pochhammer symbol. The classic one and two-parameter Mittag-Leffler functions $E_a$ and $E_{a, b}$ coincide with $E_{a,1}^1$ and $E_{a,b}^1$ respectively. In \cite{bar+al:16a} it is shown that for $\gamma \in (0,1)$ a TPL random variable $X$ has probability density function (p.d.f.) $f_X$ given by

\begin{equation}\label{eq:pdf+}
f_X(x; \gamma, \lambda, \delta, \theta)=\frac{ e^{-\theta x} x^{\gamma \delta-1} }{\lambda ^\delta} E^\delta_{\gamma, \gamma \delta}\left(\frac{\lambda \theta^\gamma-1}{\lambda} x^\gamma \right)\I_{\{x >0 \}}.
\end{equation}
In the case $\theta=0$, forcing $\gamma \in (0,1]$, this collapses to the p.d.f. of a PL law, known since  \cite{lin:63}. However, the authors observe that when instead $\gamma \in (-\infty, 0)$ the distribution is not absolutely continuous. Nevertheless, the conditional p.d.f. 
 on the event $\{X >0\}$ is available.
For more details  on PL and TPL 
families of laws 
see \cite{bar+al:16a} and \cite{bar+al:16b}.



\subsection{Infinite divisibility and L\'evy-Khintchine representation}

A key property  of the TPL distribution is its infinite divisibility. We recall that a positive random variable $X$ is said to be infinitely-divisible if for all $n=1,2 \ldots$, there exist $n$ i.i.d. random variables $X_{k,n}$, $k=1,\ldots, n,$ such that  $X=^d X_{1,n}+ \ldots + X_{n,n}$. Infinite divisibility of a positive r.v. $X$ is equivalent to require that the logarithm of $L_X$ is a Bernstein function (\citealt{sch+von:12}, Lemma 5.8). This means that there exists a positive measure $\nu$ supported on $\mathbb R_+$ such that $\int_0^\infty( 1 \wedge x) \nu(dx) <\infty$  and constants $a, b>0$ such that
\begin{equation}\label{eq:LK}
\phi_X(s):=-\log(L_X(s))=a + b s + \int_{(0, \infty)} (1-e^{-s t}) \nu(dt). 
\end{equation}
 
Equation \eqref{eq:LK} above is called the L\'evy-Khintchine decomposition of $X$ 
and $(a, b, \nu)$ is referred to as the triplet of L\'evy characteristics with L\'evy measure $\nu$,  and the Bernstein function $\phi_X$ as the characteristic (Laplace) exponent (see e.g. \citealt{sat:99}).

\smallskip

If $f$ and $g$ are Bernstein functions  then so is $f \circ g$ (\citealt{sch+von:12}, Corollary 3.8, $iii$). Therefore, if $Y$ and $Z$ are independent positive r.v.s then $\phi_Z \circ \phi_Y$ is the characteristic exponent of some positive infinitely divisible random variable $X$. Moreover, if  $(a, b,\mu)$ and $(\alpha, \beta, \rho)$ are the L\'evy characteristics triplets respectively of $Y$ and $Z$ the L\'evy triplet of $X$ is given by $(\phi_Z(a), b  \beta, \eta)$, where for all Borel sets $B$
\begin{equation}\label{eq:Bochner}
\eta(B)=\int_{(0,\infty)} \mu^Y_t(B) \rho(dt) + b \mu(B),
\end{equation} 
where $(\mu^Y_t)_{t \geq 0}$ is the convolution semigroup of probability measures associated to $\phi_Y$ (\citealt{sch+von:12}, Theorem 5.27). Equation \eqref{eq:Bochner} has the statistical interpretation of $X$ being a mixture of $Y$ over the mixing density $Z$, and the dynamic interpretation of a subordination of increasing L\'evy processes  (\citealt{sat:99}, Chapter 30).

\smallskip

 We recall that for $(\lambda, \delta) \in \mathbb R^2_+$ the characteristic exponent $\phi_Z$ and L\'evy density $u_Z$ of a gamma  $G(\lambda, \delta)$ r.v. $Z$, whose p.d.f. is given by 
\begin{equation}\label{eq:gammapdf}
f_Z(x; \lambda, \delta)= \frac{x^{\lambda-1}}{\Gamma(\lambda)\delta^\lambda} e^{-x/\delta} \I_{\{x > 0\}}, \end{equation}
   are respectively
\begin{align}\label{eq:gammaCE}
\phi_Z(s)&=\delta \log(1+ \lambda s ), \qquad  \text{Re}(s)>-1/\lambda, \\ u_Z(x)&=\delta \frac{e^{-x/\lambda}}{x}\I_{\{x>0 \}} .
\end{align}

The characteristic exponent and L\'evy density of a $\text{TPS}(\gamma, \lambda,\theta)$ r.v. $Y$ for $(\gamma, \lambda, \theta) \in \pi^{\delta}(S)$, where $\pi^{\delta}$ is the projection on the $\delta=0$ subspace of $\mathbb R^4$, is: 
\begin{align}\label{eq:TPSce}
\phi_Y(s)&=\text{sgn}(\gamma) \lambda( (\theta+s)^\gamma -\theta^\gamma), \qquad \text{Re}(s)>0, \\ u_Y(x)&=\frac{|\gamma|\lambda}{\Gamma(1-\gamma)}\frac{e^{-\theta x}}{x^{\gamma+1}} \I_{\{x>0\}}
\end{align}
and for $\gamma \in (0,1)$, the r.v. $Y$ admits the series representation 
\begin{equation}\label{eq:TPSpdf}
f_{Y}(x; \gamma, \lambda, \theta)  =\frac{e^{-\theta x +  \lambda  \theta^\gamma}}{x}
\, \sum_{k=1}^\infty\frac{1}{k!\Gamma(-k\gamma)}\,
\left(-\frac{x^\gamma}{ \lambda }\right)^{-k}\I_{
\{x>0\}},
\end{equation}
which can be obtained by exponentially tilting with parameter $\theta>0$ the series representation  given in \citet[p.88]{sat:99} of the PS $(\gamma,\lambda)$ law.


\medskip

The following result has been established in \cite{bar+al:16a} by considering limits of the probability measures as $t$ tends to zero. For $\gamma \in (0,1)$ a proof using \eqref{eq:Bochner} is offered in \cite{kum+al:19a}.  
 In the Proposition below, we summarize these results and extend  them   to the case $\gamma <0$. 

\begin{prop}\label{LevyTPL}
Let   $Y$ and $Z$ be independent r.v.s distributed respectively according to a $\text{\upshape{ TPS}}(\gamma, \lambda, \theta)$  and a $G(1, \delta)$ law. Then the r.v. $X$ whose characteristic exponent $\phi_X$ is given by  
\begin{equation}\label{eq:comp}
\phi_X(s)=\phi_Z(\phi_Y(s))  
\end{equation}
has $\text{\upshape{TPL}}(\gamma, \lambda, \delta, \theta)$ distribution. As a consequence any {\upshape TPL} r.v. $X$ is infinitely-divisible with triplet $(0,0, \nu)$, where $\nu$ is an absolutely-continuous measure. Furthermore if  \begin{equation} 
(\gamma, \lambda,\delta , \theta)\in \{ \gamma <0  \} \cup \{0<\gamma<1 ,  \lambda \theta^\gamma<1  \} \subset  S
 \end{equation}
 then $\nu$ has density $u_X$ given by

\begin{equation}\label{eq:LM}
u_X(x)=|\gamma|\delta \frac{e^{-\theta x}}{x}\left(E_{|\gamma|} \left( \c \, x^{|\gamma|} \right) - \I_{\{\text{\upshape{sgn}}(\gamma)= -1\}} \right)\I_{\{x>0\}} 
\end{equation}
with

\begin{equation}\label{c}
 \c:= \left(\frac{\lambda \theta^\gamma-\text{\upshape{sgn}}(\gamma)}{\lambda}\right)^{\text{\upshape{sgn}}(\gamma)} .
\end{equation}

\end{prop}

\begin{proof}

 
Equation \eqref{eq:comp} is straightforward from \eqref{eq:TPLdef}, \eqref{eq:gammaCE} and \eqref{eq:TPSce},   and  $X$ is infinitely-divisible because both $Z$ and $Y$ are. 



\noindent Let us now assume $\gamma \in (0,1)$. 
We can  apply \eqref{eq:Bochner} with $\mu_t^Y$ being the absolutely-continuous measures given by the densities $f_Y(x; \gamma, t  \lambda, \theta)$ from \eqref{eq:TPSpdf} and  the L\'evy triplet $(0,0, \delta \frac{e^{-t}}{t} \I_{\{t>0\}}dt)$ characterizing $Z$. Whenever $\lambda\theta^\gamma <1$ this gives for $x>0$ a uniformly-integrable series, 
 which we can integrate term by term to get the following L\'evy density for $X$

\begin{align}\label{eq:gamma+series}
u_X(x)&=\delta \frac{e^{-\theta x}}{x} \sum_{k=1}^\infty\frac{1}{k!\Gamma(-k\gamma)}\, \left(-\frac{x^\gamma}{\lambda}\right)^{- k}
 \int_0^\infty t^{k-1} e^{\left( \lambda\theta^\gamma-1 \right)t}  dt \nonumber \\&=\delta  \frac{e^{-\theta x}}{x} \sum_{k=1}^\infty\frac{1}{k \Gamma(-k\gamma)}\, \left( \frac{\lambda \theta^\gamma-1 }{\lambda}x^{\gamma}\right)^{-k}  \nonumber  \\&=-\gamma \delta  \frac{e^{-\theta x}}{x} \sum_{k=1}^\infty\frac{1}{\Gamma(1-k\gamma)}\, \left( \c \, x^{\gamma}\right)^{-k} \nonumber 
 \\&=\gamma \delta  \frac{e^{-\theta x}}{x} E_{\gamma}\left( \c \, x^\gamma \right)
\end{align}
after having applied \citet{hau+al:11}, Equation (9.2), in the second to last line.

\smallskip

\noindent If instead $\gamma <0$ we observe that we can rewrite

\begin{equation}\label{eq:CPPY}
\phi_Y(s)=\lambda \theta^\gamma \left(1-\left(\frac{1}{1+ s/\theta } \right)^{-\gamma}    \right)=\lambda \theta^\gamma (1-e^{-\phi_{Z}(s) }),
\end{equation}
where $Z$ has law $G(-\gamma, 1/\theta)$. Therefore $Y$ is in distribution a compound Poisson process  with L\'evy density  $\lambda \theta^\gamma f_{Z}$ where $f_{Z}$ is the p.d.f. of $Z$, and the measures $\mu^Y_t$ are the  laws of $Y$ with intensity $\lambda \theta^\gamma$ and i.i.d. excursions $Z$. It is well-known (e.g. \citealt{sat:99}) that $\mu^Y_t$ have the Lebesgue decomposition, for any Borel set $B$: 

\begin{align}\label{eq:CPP}
\mu^Y_t(B)&=e^{-  \lambda   \theta^\gamma t}\delta_0(B)+\sum_{k=1}^\infty  \frac{e^{- \lambda \theta^\gamma t}}{k!} (  \lambda \theta^\gamma t)^k \int_B f^{*k}_{{Z}}(x)dx \nonumber \\  &= e^{-   \lambda \theta^\gamma t}\delta_0(B)+\sum_{k=1}^\infty  \frac{e^{-  \lambda \theta^\gamma t }}{k!}  (t \lambda) ^k \int_B   \frac{x^{-\gamma k-1 } } {\Gamma(-\gamma k)}e^{-   \theta x }  dx
\end{align}
with the usual convolution notation and where $\delta_0$ is the Dirac distribution concentrated in 0. According to \eqref{eq:Bochner} we have the uniformly integrable series:

\begin{align}\label{eq:gamma-series}
u_X(x)&=\delta \frac{e^{-  \theta x }}{x} \sum_{k=1}^\infty \frac{ x^{-\gamma k } }{k! \Gamma(-\gamma k)} \lambda ^k   \int_0^\infty t^{k-1} e^{-   \left( \lambda \theta^\gamma +1   \right) t}  dt \nonumber  \\ &= \delta \frac{e^{-\theta x  }}{x} \sum_{k=1}^\infty \frac{1}{k \Gamma(-\gamma k)} \left( \frac{\lambda}{\lambda \theta^\gamma+1} x^ {-\gamma} \right)^k    \nonumber   \\ &= -\gamma \delta \frac{e^{-\theta x  }}{x} \left( E_{-\gamma}\left( \c \,  x^ {-\gamma} \right)    -1 \right).
\end{align}
Combining \eqref{eq:gamma+series} and \eqref{eq:gamma-series} yields \eqref{eq:LM}. 
\end{proof}

The case $\theta=0$ recovers the L\'evy measure of the PL distribution as given in e.g. \cite{bn:00}. 

Equation \eqref{eq:LM} serves as a starting point for the analysis of the processes based on a TPL a law. Furthermore,  it provides the structure of the  cumulants of a TPL distribution. The following result is new.

\begin{prop}\label{cumulantsTPL}
Under the parameters restrictions and in the notation of Proposition \ref{LevyTPL} and assuming additionally $\theta \neq 0$, for $n \in \mathbb N$ 
let $\kappa^+_n$ and $\kappa^-_n$ be the cumulants  of the {\upshape TPL} distribution respectively in the regimes $\gamma \in (0,1)$ and $\gamma \in (-\infty, 0)$. We have

\begin{equation}\label{eq:equation}
\kappa^\pm_n=\frac{|\gamma|\delta}{\theta^n} g_{n-1}^\pm\left( \frac{ \c }{ \theta^{|\gamma|} } \right)
\end{equation}
where $g^\pm_n(x)$ satisfy the recursion
\begin{equation}\label{eq:recurcumr}
g^\pm_{n}(x)= x |\gamma| \frac{d}{dx}g^\pm_{n-1}(x)+ n g^\pm_{n-1}(x).
\end{equation}
with, for $|x|<1$, \begin{equation}
g^+_{0}(x)=\frac{1}{1-x}, \qquad g^-_{0}(x)=\frac{x}{1-x}.
\end{equation}
\end{prop}

\begin{proof}
By differentiating the characteristic function,  the cumulants of a positive infinitely-divisible distribution can be seen to be given by the $n$-th moment integral (modulo adding the linear characteristic when $n=1$) of the L\'evy measure. In our case, recalling \eqref{eq:Bochner}, we have when $\gamma \in (0,1)$:
\begin{align}
\kappa^+_n & =\int_0^\infty x^n u_X(x) dx=\gamma \delta \int_0^\infty e^{-\theta x}\sum_{k=0}^\infty\frac{1}{\Gamma(k\gamma+1)}\, \c^{k} x^{\gamma k +n-1} dx \nonumber \\&=
\gamma \delta  \sum_{k=0}^\infty \frac{1}{\Gamma(k\gamma+1)}\,\c^{k} \int_0^\infty   x^{\gamma k+n-1} e^{-\theta x} dx \nonumber \\ &=  \frac{\gamma \delta}{\theta^n}  \sum_{k=0}^\infty \frac{\Gamma(k\gamma+n)}{\Gamma(k\gamma+1)}\,\left( \frac{\c }{\theta^\gamma}\right)^{k} = \frac{\gamma \delta}{\theta^n} \sum_{k=0}^\infty \left( \frac{\c }{\theta^\gamma}\right)^{k} (k \gamma +1)_{n-1} 
 \end{align}
which is a convergent series. The generating function of $g^+_n(x)= \sum_{k=0}^\infty (k \gamma +1)_{n} x^{k}$ of the sequence $a_k=(\gamma k+1)_n$ when $n$ is fixed can be treated as follows
\begin{align}\label{eq:pochrecur}
\sum_{k=0}^\infty (k \gamma +1)_{n} x^{k}&=
  \sum_{k=0}^\infty k \gamma (k \gamma +1)_{n-1} x^{k} +  n  \sum_{k=0}^\infty (k \gamma +1)_{n-1} x^{k}  \nonumber \\& = \gamma x  \frac{d}{dx}\left( \sum_{k=0}^\infty (k \gamma +1)_{n-1} x^{k} \right) + n \sum_{k=0}^\infty (k \gamma +1)_{n-1} x^{k}  \end{align}
Since $g^+_0(x)=1/(1-x)$,  \eqref{eq:recurcumr} follows in the positive $\gamma$ regime. 


\noindent If instead $\gamma <0$ we have
\begin{align}
\kappa^-_n & =\int_0^\infty x^n u(x) dx=-\gamma \delta \int_0^\infty e^{-\theta x}\sum_{k=1}^\infty\frac{1}{\Gamma(-k\gamma+1)}\,  \c ^{k} x^{-\gamma k +n-1} dx \nonumber \\&=
-\gamma \delta  \sum_{k=1}^\infty \frac{1}{\Gamma(k\gamma+1)}\, \c ^{k} \int_0^\infty   x^{-\gamma k+n-1}e^{-\theta x} dx \nonumber \\&= -\frac{\gamma \delta}{\theta^n}  \sum_{k=1}^\infty \frac{\Gamma(-k\gamma+n)}{\Gamma(-k\gamma+1)}\,\left(\c \, \theta^\gamma \right)^{k} = \frac{(-\gamma) \delta}{\theta^n} \sum_{k=1}^\infty \left( \c \,\theta^\gamma \right)^{k} (-k \gamma +1)_{n-1}
 \end{align}
and the series again converges for all the admissible parameters values. Setting $g^-_n(x)= \sum_{k=1}^\infty (-k \gamma +1)_{n} x^{k}$  applying \eqref{eq:pochrecur} and observing $g^-_0(x)=x/(1-x)$ completes the proof.
\end{proof}

 From \eqref{eq:equation} we find the mean and variance  of a TPL r.v. $X$ to be
 
\begin{equation}
\mathrm E[X]=|\gamma|\delta \lambda \theta^{\gamma-1}, \qquad \mathrm{Var}[X]=\frac{\mathrm E[X]}{\delta} \left( \frac{1-\gamma}{\theta} + \mathrm E[X] \right)
\end{equation} 
which 
 correspond to those calculated in \cite{bar+al:16a} (albeit in a different parametrization).
 
\subsection{Self-decomposability and geometric infinite divisibility}\label{sec:sdgid}

A further property of an infinitely-divisible distribution is self-decomposability. A random variable $X$ is said to be self-decomposable if for all $\alpha \in (0,1)$ there exists a r.v. $X_\alpha$ independent from $X$ such that $X=^d\alpha X+X_\alpha$. 
 A self-decomposable distribution is known to be infinitely divisible, and absolutely continuous  with an absolutely continuous L\'evy density \citep{ste+vh:03}. Several  stochastic processes can be canonically constructed starting from a self-decomposable law, something that we shall exploit in Section \ref{sec:inhom} once the self-decomposable nature of a TPL r.v. is established.

Another property stronger than infinite divisibility is geometric infinite divisibility  introduced by \cite{kle+al:84}. A random variable $X$ is said to be geometrically infinitely-divisible (g.i.d.) if for any $p \in (0,1)$,  there exists a geometric random variable $G_p$ with probability mass function (p.m.f.)  \begin{equation}P(G_p=k)=p^{k-1}(1-p), \qquad k=1, 2, \ldots, \end{equation}   and i.i.d. r.v.s $Z_{n,p}$, $n = 1, 2, \ldots,$ such that
\begin{equation}\label{eq:gid}
X=^d\sum_{n=1}^{G_p} Z_{n,p}.
\end{equation}
For example a 
Mittag-Leffler random variable is g.i.d., 
as shown in  \cite{lin:98} Remark 2. 
 For other properties of the g.i.d. random variables see \cite{kle+al:84}, \cite{kal:97} and \cite{koz+rac:99}.

\begin{prop}\label{selfecomp}
 A {\upshape TPL}$(\gamma, \lambda, \delta, \theta)$ r.v. with $\gamma \in (0,1]$ is self-decomposable. A {\upshape TPL}$(\gamma, \lambda, 1, \theta)$ random variable is g.i.d. for all admissible values of $\gamma$.
\end{prop}
\begin{proof}

Regarding self-decomposability, according to \cite{ste+vh:03}  Proposition V.2.14, it is sufficient to check the case $\theta=0$, i.e. to show self-decomposability of a PL law. This is well-known (see e.g. \citealt{chr+sch:01}, Section 1.2). 

\smallskip

\noindent To show geometric infinite-divisibility observe that 
 as observed by \cite{kle+al:84}, Theorem 2, a distribution is g.i.d. if and only if its characteristic function $\psi_X(z)$ is such that  $1-1/\psi_X(z)$ is a characteristic Fourier exponent $\phi_R(-i z)=-\log(L_R(-i z))$, $z \in \mathbb C$, of an infinitely-divisible r.v. $R$. However, if $X
$ has distribution \text{TPL}$(\gamma, \lambda, 1, \theta)$ we have  \begin{equation}
1-\frac{1}{L_X(-iz)}=\text{sgn}(\gamma)\lambda((\theta - iz )^\gamma-\theta^\gamma)
\end{equation}
which is the Fourier characteristic exponent of a TPS$(\gamma, \lambda,\theta)$ law. 
\end{proof}
If $\gamma<0$ then a TPL r.v. is not self-decomposable since it is not absolutely continuous. The proposition above together with \eqref{eq:tilt} clarifies the interpretation of TPL laws as ``geometric'' analogues of TPS laws, or their ``geometric versions'', in the terminology of \cite{san+pil:14}. 

In analogy with stable laws one may wish to investigate the stability condition for $p \in (0,1)$ and some $\alpha>0$
\begin{equation}\label{eq:geomstable}
X=^d p^{1/\alpha}\sum_{n=1}^{G_p} Z_{n,p}
\end{equation}
(see \citealt{kal:97}) where $X$ has $Z_{1,p}$ distribution. A r.v. satisfying \eqref{eq:geomstable} is said to be geometrically strictly stable (\citealt{kle+al:84}, Definition 2). Applying  \cite{lin:98} Remark 2 shows that PL$(\gamma, \lambda, 1)$ (i.e. Mittag-Leffler) r.v.s satisfy  \eqref{eq:geomstable}. However, a  TPL$(\gamma, \lambda, 1, \theta)$ r.v. with $\theta >0$ does not. Indeed \eqref{eq:geomstable} is very close to characterizing Linnik distributions, and does in fact characterize symmetric or positive ones (\citealt{lin:94},    \citealt{lin:98}).

\medskip




\section{TPL L\'evy subordinators and their representation}\label{sec:levy}


Being the set of the TPL distributions an infinitely-divisible class, by the general theory for a given TPL$(\gamma, \lambda, \delta, \theta)$ law on  $(\Omega,\mathcal{F}, P)$ there exists a unique in law increasing L\'evy process (L\'evy subordinator) $X=(X_t)_{t \geq 0}$ 
supported on some filtered probability space $(\Omega, \mathcal F, \mathcal F_t, P)$
 such that $X_1$ has one such prescribed law. 
  Furthermore
\begin{equation}
L_{X_t}(s)=\E[e^{-s X_t}]=e^{-t\phi_X(s)}
\end{equation}
and therefore the r.v.s $X_t$ have    \text{TPL}$(\gamma, \lambda, t \delta, \theta) $ distribution.

The Laplace exponent of the L\'evy subordinator is by definition the Laplace exponent of its unit time marginal. Henceforth,   when we refer to a L\'evy process using a distribution, we mean the L\'evy process having such distribution as unit time margin.  The L\'evy measure of a process is the L\'evy measure of its unit time margin. Unless otherwise stated, when we write equality of processes we mean equality of the finite-dimensional distributions.

\smallskip

 The TPL process $X$ enjoys a plethora of different representations. The main one is provided directly by Proposition \ref{LevyTPL} and is given by L\'evy subordination. Since the characteristic exponent of $X$ is the composition of the characteristic exponents of a gamma law and a TPS law then for all $t$:
\begin{equation}
L_{X_t}(s)=e^{-t \phi_Z(\phi_Y(s))} 
\end{equation}
and therefore by a familiar conditioning argument
\begin{equation}\label{eq:standardTC}
X_t=^dY_{Z_t}.
\end{equation}
In other words $X$ can be represented as a TPS$(\gamma, \lambda, \theta)$ process $Y=(Y_t)_{t \geq 0}$ (a tempered stable subordinator) subordinated to an independent $G(1, \delta)$ subordinator $Z=(Z_t)_{t \geq 0}$.  We indicate subordination of a process $Y$ to $Z$ with $X=Y_Z$.
\smallskip

\begin{oss}\label{equivrep}
Associating differently the scale parameter, a fully equivalent representation in distribution for the TPL process is of the form $X= Y'_{Z'}$ where $Y'$ is a TPS$(\gamma, 1, \theta)$ process and $Z'$ a $G(\lambda, \delta)$ independent subordinator. 
 \end{oss}


From \eqref{eq:TPSce} we have that $\int_0^\infty \nu^Y(dx)=\infty$ or $\int_0^\infty \nu^Y(dx)<\infty$ depending on whether $\gamma <0$ or $\gamma \in (0,1]$. 
In the latter case the process $Y$ is of finite activity, that is, $Y$ is a compound Poisson process (CPP). Furthermore, as already observed:

\begin{equation}
\phi_Y(s)=\lambda \theta^\gamma (1-e^{-\phi_{Z }(s)})
\end{equation}
where $Z$ is a $G(-\gamma, 1/\theta)$ r.v.. Thus $Y$ is a CPP  that can be written explicitly as

\begin{equation}\label{eq:TSCPP}
Y_t=\sum_{n=0}^{N_t} Z_n,
\end{equation}
where $Z_n$, $n \geq 0$ are i.i.d. with same  distribution as $Z$ and $(N_t)_{t\geq 0}$ is a Poisson process of rate $\lambda \theta^\gamma$ independent of the  $Z_n$s.

\smallskip

The case $\gamma <0$ is of particular interest for data modelling (see e.g. \citealt{bar+al:16b}). 
In such a case  $Y$ is of finite activity and the representation above holds we shall write $ Y_-$, and $Y_+$ when instead $\gamma \in (0,1]$. Correspondingly we define  $X_-$ and $X_+$. Although these L\'evy process have different path properties representation \eqref{eq:standardTC} holds in both cases.

\smallskip

\subsection{Compound Poisson representation and the logarithmic Mittag-Leffler distribution}\label{shiftedMLdis}

Since $Y_-$ is a driftless CPP process, then equation \eqref{eq:Bochner} and Fubini's Theorem imply that $X_-$ must be a CPP too. In the following we explicitly identify its structure.


In applications
 the following class of functions is of interest 
\begin{equation}\label{eq:MLfunction}
   p(x;a,b, c, \alpha, \beta)= x^{a-1} E_{\alpha, \beta}(c x^b),  \qquad x, a,b >0, \; c \in \mathbb R,\; \text{Re}(\alpha), \text{Re}(\beta) >0,
\end{equation}
which are often seen to appear in connection with  the solution of fractional differential problems.  We can express the Laplace transform $L(\cdot;x,s)$, 
Re$(s)>0$, of \eqref{eq:MLfunction} in terms of the Fox-Wright function (e.g. \citealt{wri:35}) as follows (\citealt{mat+ara:08}, equation 2.2.22):

\begin{align}\label{eq:MLfunctionFW}
   L(p(x;a,b, c, \alpha, \beta); x,s)=&\int_0^\infty e^{-s x} x^{a-1} E_{\alpha, \beta}(c x^b) dx=\sum_{k = 0}^\infty \frac{\Gamma(a k +b )}{\Gamma(\alpha k +\beta)}\frac{c^k}{s^{a k +b} } \nonumber \\ =& \frac{1}{s^{b} } \phantom{c}   _2\Psi_1 \left[ \begin{array}{cc}  (1,1) & (b,a) \\  (\beta, \alpha) & \end{array} ;   \frac{c}{s^a}\right],   \qquad |c|<|s^a|.
\end{align}
This latter expression is not always the transform of a probability function. For example in  
\begin{align}\label{eq:MLtransform}
   L(p(x;a,b, -1, a, b); x,s)=\frac{s^{a-b}}{s^a + 1},
\end{align}
a function which is pivotal to fractional calculus and its applications (e.g. \citealt{hau+al:11}),
 we have that as $s \rightarrow 0$, then $L(p(x;a,b,-1, b, 1), x,s)\rightarrow 1$ if and only if $a=b$, in which case the associated distribution is the Mittag-Leffler distribution. 

However we can exponentially 
temper \eqref{eq:MLfunction} by $\theta>0$ obtaining, with slight notation abuse,
\begin{equation}\label{eq:MLfunctiontilted}
   p(x;a,b,c, \alpha, \beta,\theta)= e^{-\theta x} x^{a-1} E_{\alpha, \beta}(c x^b),  \qquad x, a,b , \theta>0, \, c \in \mathbb R, \, \text{Re}(\alpha),\, \text{Re}(\beta) >0,
\end{equation}
which in turn after applying the shifting rule determines the Laplace transform

\begin{align}\label{eq:MLfunctiontiltedLaplace}
   L(p(x;a,b, c, \alpha, \beta,\theta ); x,s)= \frac{1}{(s+\theta)^{b} }   \phantom{c}   _2\Psi_1 \left[ \begin{array}{cc}  (1,1) & (b,a) \\  (\beta, \alpha) & \end{array} ;   \frac{c}{(s+\theta)^a}\right],   \qquad |c|<|(s+\theta)^a|.
\end{align}
As $s \rightarrow  0$ the limit of the expression of the above is always finite, so after appropriate normalization, $p(x;a,b, c, \alpha, \beta,\theta )$ determines a probability distribution. In particular if we let $a=b=\alpha$, $\beta=a+1$, $|c|<\theta^a$ we have:

\begin{align}\label{eq:LMLfunction}
   L(p(x;a,a, c, a, a+1,\theta ), x,s) &= \frac{1}{a c}\sum_{k = 0}^{\infty} \frac{1}{k+1}\left(\frac{c}{(\theta+s)^a}\right)^{k+1} \nonumber \\ &= -\frac{1}{a c }{\log \left(1-  c (s+\theta)^{-a} \right) .} 
\end{align}
Therefore,  defining the normalizing constant \begin{equation}
n(a,c, \theta)=-\frac{ a c}{\log \left(1-  c \theta^{-a}\right)},
\end{equation}
we conclude that $n(a,c, \theta) p(x;a,a, c, a, a+1,\theta )$ is a p.d.f. of known Laplace transform. Calculating this product explicitly we can introduce the following probability distribution.

\begin{defn} The logarithmic Mittag-Leffler  (LML) distribution is the absolutely-continuous family of distributions {\upshape LML}$(a, c, \theta)$, $a,  \theta>0$, $|c|<\theta^a$,  such that an LML r.v. $W$ has p.d.f.
 \begin{align}\label{eq:LMLdef}
f_W(x;a,c,\theta)  = -\frac{ a c}{\log \left(1-  c \theta ^{-a}\right)}e^{-\theta x } x^{a-1} E_{a, a+1}(c 
x^a) \I_{\{x>0 \}}
 \end{align}  
 and Laplace transform
 \begin{equation}
 L_W(s)= \frac{\log \left(1-  c (s+\theta)^{-a} \right)}{ \log \left(1-  c \theta ^{-a}\right)}, \qquad \text{{\upshape Re}}(s)>0.
 \end{equation}
\end{defn} 
 

\medskip

The terminology is motivated by the similitude of the Laplace transform of this distribution with that of the discrete logarithmic probability law. This distribution is not a generalization of the Mittag-Leffler law as it does not admit it as a particular case, nor does it admit the degenerate case $\theta=0$.  Hence it is structurally different from a TPL law. As it turns out, the LML law arises naturally in the CPP structure of $X_-$. 

\begin{prop}\label{CPPTPL}The process $X_-$ admits the {\upshape CPP} representation

\begin{equation}
X_{-_{\, t}}=\sum_{n=0}^{N_t}J_n
\end{equation}
where $(N_t)_{t\geq 0}$ is a Poisson process  of rate
$ \delta \log(1+ \lambda \theta^\gamma)$ 
and $(J_n)_{n \geq 0}$, is an i.i.d. sequence of random variables having {\upshape LML}$(-\gamma, \c, \theta)$ distribution, where $\c$ is given by \eqref{c}.

\end{prop}

\begin{proof} Because of \eqref{eq:Bochner} $X_-$ is  driftless,  being the gamma subordination of  the driftless CPP in \eqref{eq:TSCPP}.  Using \cite{hau+al:11} Theorem 5.1 in the negative determination of \eqref{eq:LM}, we have the equivalent expression for the L\'evy density of $X_-$:

\begin{equation}\label{eq:newu}
u_{X_-}(x)=-\gamma \delta \c  e^{-\theta x}x^{-\gamma-1} E_{-\gamma,-\gamma+1}(\c x^{-\gamma})\I_{\{ x>0\}}.
\end{equation}
But $\c\theta^\gamma<1$ 
so that from  \eqref{eq:newu} and \eqref{eq:LMLdef} we have  
\begin{equation}
u_{X_-}(x) =\delta \log(1+ \lambda \theta^\gamma)f_W(x;-\gamma, \c, \theta)\I_{\{ x>0\}}
\end{equation} 
which proves the proposition.
\end{proof}

Another type of tilted Mittag-Leffler distribution following the construction outlined in this section will appear in Section \ref{sec:inhom}.


\subsection{Stochastic self-similarity and negative binomial subordination}

The TPL processes subordinated structure is extremely rich and  and, for reasons which we shall shortly explore, mostly revolves around the negative binomial subordinator. We introduce a lattice-valued version of the L\'evy subordinator $B=(B_t)_{t\geq 0}$ with unit time distribution in the family of laws NB$(\pi, \kappa, \alpha, \mu)$  given by the  Laplace transform
\begin{equation}\label{eq:NBlaplace}
L_{B}(s)= \left(\frac{\pi }{1-(1-\pi) e^{-\alpha s}} \right)^\kappa e^{- \mu s}, \qquad \pi \in (0,1), \kappa, \alpha >0, \mu \in  \mathbb R, \text{Re}(s)>0.
\end{equation}
The above law is a scale-location modification of the  negative binomial law, and it thus gives raise to an infinitely-divisible distribution. Taking the logarithm of $L_{B}$ and considering the corresponding characteristic exponent we see that $B$ is such that $B_t$ has NB$(\pi, \kappa t, \alpha, \mu)$ distribution and it can be represented as a CPP with drift as follows:

\begin{equation}
B_t=\sum_{n=0}^{N_t} J_n + \mu t
\end{equation}
with the $J_n$ being i.i.d distributed r.v.s with lattice-valued logarithmic probability mass function 
\begin{equation}\label{eq:logdis}
 P(J_n= \alpha  k)=\frac{(1-\pi)^k}{- k \log \pi}, \qquad k = 1,2, \ldots,
\end{equation}
and $N=(N_t)_{t\geq 0}$ is an independent Poisson process of intensity $-\kappa \log \pi$.  For these and other properties of the  negative binomial subordinator 
see \cite{koz+pod:09}.

Negative binomial processes appear naturally in connection to the concept of stochastic self-similarity introduced in
\cite{koz+al:06}, Definition 4.1. Let $X=(X_t)_{t\geq 0}$ be any stochastic process and assume that there exists a family of processes $T^c=\{(T^c_t)_{t\geq 0}, c>1 \}$ almost surely increasing and diverging as $t \rightarrow \infty$ such that
\begin{equation}\label{eq:sss}
X_{T^c_t} =^d c^H X_t
\end{equation}
 for some $H>0$. Then $X$ is said to be  stochastically self-similar of index $H$ with respect to $T^c$.



Stochastic self-similarity is intimately related to geometric infinite-divisibility and in particular to the stability property. Based on this relationship we establish a general invariance property of the PL processes which extends \cite{koz+al:06}, Proposition 4.2, and which in particular implies that Linnik processes are stochastically self-similar with respect to families of negative binomial processes
  (see also \citealt{bn+al:01}, Example 2.2).
 
 \begin{prop}\label{sssTPL}Let $X=(X_t)_{t \geq 0}$ be a {\upshape PL}$( \gamma, \lambda, \delta)$ L\'evy process, let $H>0$ and let $B^c=\{(B^{c}_t)_{t\geq 0}, c > 1 \}$ be a family of  {\upshape NB}$(c^{-H}  , \kappa, 1/\delta, \kappa/\delta )$ subordinators independent of $X$. Then  $Y= X_{B^c}$ 
is a {\upshape PL}$(\gamma, \lambda c^{H}, \kappa  )$ process for all $c>1$. In particular, $X$ is stochastically self-similar for all indices $H>0$ with respect to $B^c$.
 \end{prop}
 
 \begin{proof}
 The first claim is equivalent to $\phi_{B^c}( \phi_X(s))=\phi_{Y}(s c ^H)$ for all Re$(s)>0$ and $c >1$. Composing the exponents and using \eqref{eq:NBlaplace} 
\begin{align}
\phi_{B^c}(\phi_X(s))&=- \kappa\log\left( \frac{ c^{- H } e^{- \phi_X(s)/\delta}  }{1 -(1-  c^{-H })e^{ -\phi_X(s)/\delta}} \right )=- \kappa\log\left( \frac{ c ^{-H }  (1+ \lambda s ^\gamma )^{-1}  }{1 -(1- c^{-H })(1+ \lambda s^ \gamma  )^{-1}}\right) \nonumber \\ &= \kappa \log\left( 1 - c^{H}  + c^{ H }  (1+  \lambda s  ^{\gamma} ) \right)= \kappa \log\left( 1 +  c^H  \lambda   s
^{\gamma}  \right) 
\end{align}    
Stochastic self-similarity follows setting $\kappa=\delta$.
 \end{proof}

Using the subordinated structure of the TPL L\'evy process, the distributional invariance part of the proposition above can be extended to $X_+$, although stochastic self-similarity does not hold because PL and TPL processes scale differently.

\begin{cor}\label{TPLinv}
Let $X$ be a {\upshape TPL}$(\gamma, \lambda, \delta, \theta)$ L\'evy process and  $B^{\pi}=\{(B^{\pi}_t)_{t\geq 0}, \pi \in (0,1)\}$ be a family of  {\upshape NB}$( \pi  , \kappa, 1/\delta, \kappa/\delta )$ subordinators independent of $X$. 
Then $X^\pi=X_{B^\pi}$
is a {\upshape TPL}$(\gamma, \lambda \pi^{-1}, \kappa, \theta)$   L\'evy process for all $\pi \in (0,1)$.   
\end{cor}

\begin{proof}
Using Remark \ref{equivrep} we recall $X=Y_Z$ where $Y$ is a TPS${(\gamma},1, \theta)$ subordinator and $Z$ is a gamma $G(\lambda, \delta)$ process. 
Applying Proposition \ref{sssTPL} with $\gamma=H=1$, $c=\pi^{-1}$ we have that $Z^{\pi}:=Z_{B^{\pi}}$ is a  $G(\lambda \pi^{-1}, \kappa )$ gamma process.
Using independence we have the equalities 
\begin{equation}
X^\pi_t=^d (Y_Z)_{B^\pi_t}=^dY_{Z^{\pi }_t }
\end{equation}
and the conclusion follows using  again Remark \ref{equivrep} .
\end{proof}

 Unit scale negative binomial subordinators  provide an additional representation for non-degenerate ($\theta>0$) TPL L\'evy processes $X_+$ and $X_-$ as subordinated gamma processes seemingly unrelated to the results above.

\begin{prop} For $\gamma \in  (-\infty ,0) \cup (0,1)$, $\theta>0$, let $Z$ be a  gamma $G(\theta, |\gamma|)$ L\'evy process 
and $ B^{\pi}_+$ and $B_-^{\pi}$ to be two negative binomial processes independent of $Z$ respectively of unit marginals {\upshape NB}$(\pi, \delta, 1, \delta)$ and {\upshape NB}$(\pi, \delta, 1, 0)$. Then  if $\gamma \in \mathbb (0,1)$ 
\begin{equation}\label{eq:xplusNB}
X_+^{\pi}:=Z_{ B_+^{\pi}}
\end{equation}
is a {\upshape TPL} $(\gamma, \theta^\gamma \pi^{-1},\delta, \theta^{-1})$ L\'evy process. If $\gamma <0$ 
\begin{equation}
X_-^{\pi}:=Z_{B_-^{\pi}}
\end{equation}
is  a {\upshape TPL}$(\gamma,   \theta^{-\gamma}(\pi^{-1}-1) , \delta, \theta^{-1} )$ L\'evy process.
\end{prop}
\begin{proof}
 For $X_+^{\pi}$ we have
\begin{align}
\phi_{ B_+^{\pi}}(\phi_Z(s))&=-\delta \log\left( \frac{\pi e^{-\phi_Z(s)}  }{1 -(1-\pi)e^{-\phi_Z(s)}}\right)=-\delta \log\left( \frac{\pi (1+  \theta s )^{-\gamma} }{1 -(1-\pi)(1+  \theta s  )^{-\gamma}}\right) \nonumber \\& =\delta \log\left( 1 - \frac{1}{\pi} +\frac{1}{\pi}(1+  \theta s  )^{\gamma} \right)=\delta \log\left( 1 + \frac{\theta^{\gamma}}{\pi}(  ( \theta^{-1}+ s  )^{\gamma}- \theta^{-\gamma}) \right) 
\end{align}  
whereas for $X_-^{\pi}$
\begin{align}\phi_{ B_-^{\pi}}(\phi_Z(s))&=-\delta\log\left( \frac{\pi}{1 -(1-\pi)e^{-\phi_Z(s)})}\right) \nonumber =\delta \log\left(1 +\frac{1-\pi}{\pi}\left(1-(1+ \theta s)^{\gamma} \right) \right)\nonumber \\&  =\delta \log\left(1 -\theta^{-\gamma}\frac{1-\pi}{\pi}\left((\theta^{-1} +s)^{\gamma} -\theta^{-\gamma} \right) \right).
\end{align}

\end{proof}




Notice that the negative binomial representation of $X_-$ is obtained applying a driftless CPP to the infinite activity  process $Z$ which correctly determines finite activity. In contrast, the one for $X_+$ features a CPP with drift which maintains the infinite activity of the subordinand process $Y_+$.

\medskip




\subsection{A connection with potential theory}

There exists an interesting connection between gamma-subordinated L\'evy processes and the potential measure. Following \cite{sat:99}, Chapter 6, define for any Borel set $B \subset \mathbb R$ the $q$-th potential measure of a process $X=(X_t)_{t \geq 0}$ with probability laws $\mu^X_t$ as 
\begin{equation}\label{eq:qpot}
V^q(B)=\int_0^\infty e^{-q u}\mu^X_u(B) du . 
\end{equation}
Now, by \eqref{eq:Bochner} for $q>0$ the laws $\mu^{Y,q}_t$ of $Y^q:=X_{Z^q}$, where $Z^q$ is a $G(1/q,1)$ gamma process independent of $X$, write as
\begin{equation}
\mu^{Y,q}_t(B)=\frac{q^t}{\Gamma(t)}\int_0^\infty \mu^X_u(B)u^{t-1}e^{-q u} du.
\end{equation}
Clearly the law of $Y^q_1$ coincides with $q V^q $, $q>0$ . Therefore,  the knowledge of the unit time law of $Y^q$ completely determines the $q$-th potential measure of $X$. 
But TPL processes are a particular case of gamma-subordinated L\'evy process whose probability laws are known. According to the above, this means that the whole $q$-potential structure, $q>0$, of a TPS$(\gamma, \lambda, \theta)$ law can be made explicit. A simple computation shows the following: 

\begin{prop}
Let $Y$ be a {\upshape TPS}$(\gamma, \lambda, \theta)$ subordinator. The $q>0$ potential measures of $Y$ are absolutely continuous, and the potential densities $v^q(x)$ are given by
\begin{equation}
v^q(x)=e^{-\theta x} x^{\gamma-1} E_{\gamma, \gamma}\left((\theta^\gamma-q) x^\gamma \right). \end{equation}
\end{prop}

The  $0$-th potential measure (the potential measure \emph{tout court}) of tempered stable subordinators has been calculated using contour integration methods by \cite{kum:20}.



\section{Inhomogeneous TPL processes}\label{sec:inhom}

We discuss  two non-homogeneous (non-L\'evy) Markovian TPL processes: the L\'evy-driven OU process with TPL stationary distribution and the self-similar process with independent 
increments (Sato process) with unit time TPL marginal. The existence of these processes essentially stem from the self-decomposability property of the TPL distribution whenever $\gamma \in (0,1]$.

\subsection{The OU process with stationary TPL distribution}

A L\'evy-driven OU process is the solution $X=(X_t)_{t \geq 0}$ on $(\Omega, \mathcal F, \mathcal F_t, P)$
of the stochastic differential equation (SDE) 

\begin{equation}\label{eq:OUSDE}
 X_t=X_0 - \alpha \int_0^t X_u du  + \int_0^t d Z_u^\alpha   
\end{equation}
given by
\begin{equation}
 X_t= e^{-\alpha t}X_0 +   \int_0^t e^{-\alpha(t-u)} d Z_u^\alpha   
\end{equation}
for some adapted L\'evy process $Z^\alpha=(Z^\alpha_t)_{t\geq0}$, $\alpha>0$. The theory of OU L\'evy-driven SDEs and their applications is fully detailed in \cite{bn:97},  \cite{bn+al:02} and  \cite{bn+she:02}  using prior results of \cite{jur+ver:83} and \cite{wol:82}. 

\smallskip

 The law of $X$ is clearly determined by that of $Z^\alpha$. Conversely, under some conditions, for any self-decomposable distribution $D$ there exists a L\'evy process $Z^\alpha$ of known L\'evy triplet, such that  \eqref{eq:OUSDE} admits a stationary solution $X$ with law $D$.  
 
\medskip 
 

In order to study the process $Z^\alpha$ determining a TPL stationary solution $X$ to \eqref{eq:OUSDE}, we begin by introducing the tempered Mittag-Leffler (TML) distribution distribution. 
A TML distribution is obtained by exponentially tempering with $\theta >0$ the survival function of \cite{pil:90} Mittag-Leffler distribution. We illustrate such a family in the following result.%


\begin{prop}
A TML r.v. $U$ with  {\upshape TML}$(a, c, \theta)$,  distribution where $( a, c, \theta) \in (0,1 ] \times \mathbb R \times \mathbb R_+$, is a positive distribution defined by  the cumulative distribution function (c.d.f.) 
\begin{equation}\label{eq:LMLcdf}
F_U(x; a, c, \theta)=\left(1-e^{-\theta x}  E_{a}(- c x ^a)\right)\I_{\{x \geq 0\}},
\end{equation} with  p.d.f.
\begin{align}\label{eq:LMLpdf}
f_U(x; a, c,  \theta)&
= e^{-\theta x} (\theta E_{a}(- c x ^a)+ c \, x^{a-1}E_{a,a}(- c \, x ^a) )\I_{\{x \geq 0\}}  
\end{align}
and Laplace transform
\begin{equation}\label{eq:TML}
L_U(s) = \frac{\theta(s+\theta)^{a-1}+c}{(s+\theta)^{a}+c}, \qquad \text{{\upshape Re}}(s)>0.
\end{equation}
Furthermore, $U$ is infinitely-divisible.
\end{prop}
\begin{proof}
Using the properties of the Mittag-Leffler function, that $F_U$ is a positively-supported c.d.f. is clear. By differentiating  in $x$ we have 
\begin{align}\label{eq:LMLpdfproof}
f_U(x; a, c,  \theta) &=
 e^{-\theta x}\left(  \theta E_{a}\left(- cx^a \right)- x^{a-1}\sum_{k=0}^\infty  a (k+1)(- c)^{k+1} \frac{x^{a k}}{\Gamma(a k+ a+1)}  \right)  \nonumber \\
&= e^{-\theta x}\left (\theta E_{a}(- c x ^a)+ c \, x^{a-1}\sum_{k=0}^\infty \frac{(- c \, x ^a)^k }{\Gamma(a k +a)}  \right)  
\end{align}
which yields \eqref{eq:LMLpdf}. Using  \eqref{eq:MLfunctiontiltedLaplace} with the appropriate parameters on both terms in \eqref{eq:LMLpdf} we have
\begin{equation}\label{eq:TMLproof}
L_U(s)= \theta \frac{(s+\theta)^{a-1}}{(s+\theta)^{a}+c} + \frac{c}{(s+\theta)^{a}+c} 
\end{equation}
and \eqref{eq:TML} follows. To show that the TML distribution is infinitely-divisible is necessary and sufficient to show that the logarithmic derivative of $-L_U$ is a completely monotone function (e.g. \citealt{gor+al:14}, Chapter 9). But for $s>0$
\begin{align}
-\frac{d}{ds} \log(L_U(s))&=(s+\theta)^{a-1}\left(\frac{\theta(1-a)}{c(s+\theta)+ \theta(s+\theta)^{a}} +\frac{a}{c+ (s+\theta)^{a}} \right) 
\end{align}
which is a product of positive linear combinations of completely monotone functions, and hence is itself completely monotone (\citealt{sch+von:12}, Corollary 1.6).
\end{proof}

\bigskip
The TML distribution dictates the activity of the CPP process $Z^\alpha$ when $X$ is a stationary solution to \eqref{eq:OUSDE}.  The next Proposition closely mirrors Proposition \ref{CPPTPL}.





\begin{prop}\label{OUTPL}
Let $X$ have {\upshape TPL}$(\gamma, \lambda, \delta, \theta)$ with  $\gamma \in (0,1]$, and $\lambda \theta^\gamma<1$. Then $X$ is the law of the stationary solution to \eqref{eq:OUSDE} with \begin{equation}\label{eq:levydrivCPP}
Z^{\alpha}_t=\sum_{n = 0}^{N_t^\alpha}U_n
\end{equation} where $N^\alpha=(N_t^\alpha)_{t\geq 0}$ is a Poisson process of intensity $\alpha  \delta \gamma$ while $(U_n)_{n \geq 0}$ is an i.i.d. sequence of r.v.s independent of $N^\alpha$ with common distribution $\text{{\upshape TML}}(\gamma, -\c, \theta)$ and $\c$ is given by \eqref{c}.
 Moreover
\begin{equation}\label{eq:Zlcharexp}
\phi_\alpha(s):=\phi_{Z^\alpha}(s)=\alpha \delta  \gamma \frac{\lambda s ( \theta+s)^{\gamma-1}}{1+\lambda((\theta+s)^\gamma - \theta^\gamma) }.
\end{equation}


\end{prop}

\begin{proof}
According to Proposition \ref{selfecomp} whenever $\gamma \in \mathbb (0,1]$, the r.v. $X$ is self-decomposable. According to e.g. \cite{bn:97} Theorem 2.2,   it holds $\phi_\alpha(s)=\alpha s \phi'_X(s)$ so long as this latter expressions is continuous in zero. Such calculation produces \eqref{eq:Zlcharexp} and continuity is easily checked. 
Moreover, it is easy to show that
\begin{equation}\label{eq:Zlcharexp2}
\phi_\alpha(s)=\alpha \delta \gamma \left( 1- \frac{\theta ( \theta+s)^{\gamma-1} - \c}{ (\theta+s)^\gamma - \c} \right)
\end{equation}
and in the second term inside the parentheses we recognize the Laplace transform \eqref{eq:TML} with the required parameters. 
 This characterizes the law of $Z^\alpha$ as that of the CPP in \eqref{eq:levydrivCPP}.
\end{proof}

The CPP structure of the L\'evy driving noise is typical for a large class of self-decomposable distributions $D$. It is known (\citealt{ste+vh:03}, Theorem V.6.12) that $u_D(x)$ must be such that $k(x):=x u_D(x)$ is  non-increasing. On the other hand, as observed in e.g. \cite{bn+she:02}, equations (16)--(17), we have
\begin{equation}\label{eq:intuza}
\int_{x}^\infty  u_{Z^\alpha}(u)du= \alpha xu_D(x)=\alpha k(x).
\end{equation}
 Assume now that $k(x)$ is differentiable and finite in zero. From \eqref{eq:intuza} it follows 
\begin{equation}\label{eq:uza}
u_{Z^\alpha}(x)=-\alpha k'(x).\end{equation} and setting  $x=0$ in \eqref{eq:intuza} shows that $Z^\alpha$ is a CPP and the p.d.f. of the increments equals $- k'(x)/\alpha k(0+)$. This argument does provide an alternative proof of Proposition \eqref{OUTPL}: from  \eqref{eq:LM} specifying
\begin{equation}
k(x)=\gamma\delta e^{-\theta x}E_\gamma \left( \c \, x^{\gamma} \right)\I_{\{x>0\}},
\end{equation}
 differentiating and substituting in \eqref{eq:uza} recovers the L\'evy density of the CPP in \eqref{eq:levydrivCPP}. 

\smallskip

In the case $\gamma=1$ or $\lambda \theta^\gamma=1$ we fall back to two instances of the popular gamma OU L\'evy-driven model discussed in \cite{bn+she:02}, \cite{bn+al:01}  with respectively $G(\lambda,\delta)$  and $G(1/\theta,\delta \gamma)$ stationary solution. Accordingly, in such a case the TML increments reduce to exponential variables of parameter  $\lambda$ (resp. $1/\theta$).
Furthermore, we have the notable particular case $\theta=0$ in which the stationary OU solution with PL distribution has Mittag-Leffler  driving noise. The PL stationary OU L\'evy-driven process can be thus seen as the natural modification of a gamma stationary OU process upon introduction of the $\gamma$ tail parameter.

\medskip

The OU representation of a stationary TPL process is very well-suited for numerical schemes of Euler type, where the innovation $U_i$ can be treated by inverse-CDF sampling using equation \eqref{eq:LMLcdf}. We exemplify this in Figure \ref{fig:ouplot} where we simulate both the stationary gamma OU L\'evy-driven model and its TPL counterpart with same random variate drawings. The former is attained from the latter by using same parameters but changing $\gamma$ to 1. In the TPL model $\gamma$ and $\theta$ govern the tail of the TML jumps: the smaller such parameters the biggest the incidence of large upward jumps in the OU process, a feature which is particularly appealing for modeling financial returns volatility.

\begin{figure}
\begin{floatrow}
\includegraphics[height=8cm, width=12cm]{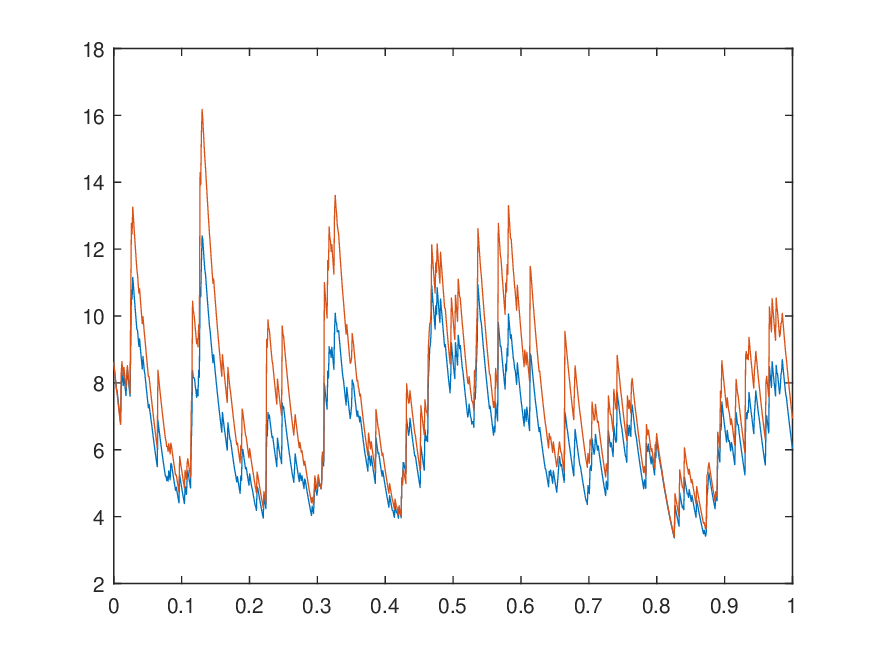}}
{\caption[]{\small{In blue a trajectory  of the OU gamma stationary model and in red a trajectory of its TPL extension from the same random drawing. The parameters are $\alpha=25$, $\delta=20$, $\lambda=\theta=0.5$, $\gamma=0.7$. }} \label{fig:ouplot}
\end{floatrow}
\end{figure}

Another classic application of L\'evy driven SDEs is the explicit construction of a stationary process  with  (quasi) long-range dependence, which can be attained using a superposition of the SDEs \eqref{eq:OUSDE}  as explained in   \cite{bn:97},  Theorem 4.1 .  






\subsection{Self-similar TPL processes with independent increments}

As shown  by \cite{sat:91},  for all $H>0$ to any self-decomposable distribution $D$ we can associate a self-similar process with Hurst exponent $H$ with independent increments (s.s.i.i.) having $D$ as unit time marginal. Unless $D$ is a stable distribution, such process will not be the same one as the L\'evy process with unit time law $D$. As it turns out, when $D$ is TPL all the marginals of the TPL s.s.i.i. process remain TPL and we thus have an explicit representation for its law and L\'evy measure.

\begin{prop}\label{SatoTPL} Let $H>0$ and $X$ be a {\upshape TPL}$(\gamma, \lambda, \delta, \theta)$ r.v. with $\gamma \in (0,1)$. There exists a stochastically-continuous s.s.i.i. process $X^H=(X_t^H)_{t \geq 0}$  of Hurst index $H$ 
with independent increments such that $X_1^H$ has the same distribution as $X$, and whose triplet of the integrated semimartingale characteristics 
is 
$(0,0,U^H(dt, dx))$ with  $U^H(dt, dx)$ having density
\begin{equation}\label{eq:INTLM}
u^H_X(t,x)= \gamma \delta \frac{e^{-\theta t^{-H} x}}{x}E_{\gamma}  \left( \c ( t^{-H } x)^\gamma \right). 
\end{equation}
In particular
$X^H_t $ has $ \text{{\upshape TPL}}(\gamma, \delta, \lambda t^{H \gamma}, \theta t^{-H})$ distribution.
Furthermore we have the subordinated representation
\begin{equation}\label{eq:SSTC}
X^H=Y^H_Z
\end{equation}
where $Y^H$ is the s.s.i.i. process associated to a {\upshape TPS}$(\gamma,\lambda, \theta)$ law
which is such that $Y_t^H$ has {\upshape TPS}$(\gamma, \lambda t^{H \gamma}, \theta t^{-H})$ distribution,
and $Z$ is a $G(1, \delta)$ independent gamma L\'evy process.
\end{prop}

\begin{proof}
The existence of $X^H$ for a given unit time  self-decomposable marginal $X$, and its characterization in terms of the integrated semimartingale characteristic triplet is provided in \cite{sat:91}. In particular  the integrated L\'evy measure of $X^H$ is absolutely continuous and its density is given by 
\begin{equation}\label{eq:satsub}
 u^H_X(t,x)=t^{-H}u_X(t^{-H} x).
\end{equation}
Remembering \eqref{eq:LM} and the density  \eqref{eq:INTLM}   follows. 

To prove the second statement consider the s.s.i.i. process with independent increments $Y_t^H$ and combine the substitution in \eqref{eq:satsub} with the L\'evy density \eqref{eq:TPSce} which determines the integrated L\'evy density $u^H_Y(t,x)$ of $Y_t^H$ as
\begin{equation}
u^H_Y(t,x)= \frac{\gamma\lambda t^{-H}}{\Gamma(1-\gamma)}\frac{e^{-\theta t^{-H} x}}{{(xt^{-H})}^{\gamma+1}} \I_{\{x>0\}} dx=\frac{\gamma\lambda t^{\gamma H}}{\Gamma(1-\gamma)}\frac{e^{-\theta t^{-H} x}}{{x}^{\gamma+1}} \I_{\{x>0\}} dx
\end{equation}
showing  that $Y^H_t$ has  TPS$(\gamma, \lambda t^{H \gamma}, \theta t^{-H})$ law.
Now  using \eqref{eq:satsub} in \eqref{eq:comp} implies
\begin{equation}
\phi_{X^H}(s)=\phi_X (s t^H)=\phi_Z(\phi_Y(s t^H))=\phi_Z(\phi_{Y^H}(s))
\end{equation}
which terminates the proof. 
\end{proof}


\smallskip

We observe that  as $t \rightarrow \infty$ tempering tends to zero and $X^H_t$  approaches a large scale PL variable. 


Self-similarity is a property which is often observed in financial returns  time series. Using additive processes in place of L\'evy ones in finance has also benefits for valuation of derivative securities. It is recognized that normalized cumulants of risk-neutral distributions implicit in option prices do not decrease with time to expiration of contracts, or at least not as rapidly as the linear rate of decay predicted by L\'evy process, a behaviour which is corrected by removing the assumption of returns stationarity. 




\section{Multivariate TPL processes}\label{sec:multivariate}
Stochastic self-similarity with respect to the negative binomial subordinator of the gamma process can be exploited for generating multivariate TPL L\'evy processes in a natural way, which we illustrate in the following. 
Multivariate g.i.d. laws are studied in \cite{mit+rac:91}: recently, multivariate  Mittag-Leffler distributions have been explored in \cite{alb:20}  and \cite{kho+al:20}. 

\smallskip

Let $d \in \mathbb N$ and $X=(X_t)_{ t\geq 0}$ with 
$
X_t=( X^1_t, \ldots, X_t^d)
$ be an independent  multivariate  TPL$(\gamma_i, \lambda_i , 1, \theta_i)$
 L\'evy process i.e. $X$ is such that  for all $t$, $X^i_t$ is independent from $X^j_t$ whenever $i \neq j$. 
  Let  $B^\pi$ be a negative binomial process NB$(\pi, \delta, 1, \delta)$ independent of $X$. According to Corollary \ref{TPLinv} the subordinate multivariate L\'evy process  $X^\pi=(X^\pi_t)_{t \geq 0}$ with
   \begin{equation}
  X^\pi_t=( X^{\pi,1}_t, \ldots, X_t^{\pi,d}):=(X^1_{B^\pi_t}, \ldots, X^d_{B^\pi_t})  \end{equation} is such that
$X^{\pi,i}=(X^{\pi,i})_{t\geq 0}$ has  $\text{TPL}(\gamma_i,   \lambda_i \pi^{-1} , \delta, \theta_i)$ law. Therefore $ X^\pi$ is a multivariate L\'evy  process with correlated TPL marginals, conditionally independent on $B^\pi$, and the success probability $\pi$ plays the role of a dependence parameter with the degenerate case $\pi=1$ amounting to the independent case ($B^\pi$ being pure drift in such a case). 
According to the general properties of TPL laws illustrated in Section \ref{sec:dist}, depending on whether $\gamma_i \in (0,1]$ or $\gamma_i<0$ the marginal processes can be either infinite activity with nonintegrable L\'evy marginal measure and absolutely-continuous law, or CPPs, whose law has a point mass in zero.

\smallskip
A useful alternative representation of $X^\pi$ can also be provided. By virtue of Remark \ref{equivrep}   for all $i=1, \ldots, d$, we can interpret the marginal processes $X^i$ as  $X^i= Y^i_{Z^i}$, for two independent multivariate L\'evy processes $Y^i$ and $Z^i$, where $Y^i$ is a TPS$(\gamma_i, 1,\theta_i)$ process and  $Z^i$ is a  gamma $G(\lambda_i, 1)$ process independent of $Y^i$ and therefore
\begin{equation}\label{xmulti}
X_t=^d\left(Y^1_{Z^1_t}, \ldots, Y^d_{Z^d_t}\right) .
\end{equation}
Choosing further  $Y^j$, $Y^i$ and $Z^i$, $Z^j$ to be independent whenever $i \neq j$, we can introduce two independent multivariate 
L\'evy processes $Y=(Y^1_t, \ldots , Y_t^d)$ and $Z=(Z^1_t, \ldots , Z_t^d)$ with independent marginals and \eqref{xmulti} has the interpretation of a multivariate subordination of $Y$  to  $Z$, as detailed \cite{bn+al:01}. We shall denote multivariate subordination in the same way as the standard one, and therefore \eqref{xmulti} implicates $X=Y_Z$.
 Furthermore, by Proposition \ref{sssTPL} it holds
\begin{equation}\label{xpimulti}
X^{\pi}_t=^d\left(\left(Y^1_{{Z^1}}\right)_{B^\pi_t}, \ldots, \left(Y^d_{{Z^d}}\right)_{B^\pi_t}\right)  =^d\left(Y^1_{Z^{\pi,1}_t}, \ldots,Y^d_{Z^{\pi,d}_t}\right), 
\end{equation}
where $Z^{\pi,i}=Z^{i}_{B^\pi}$ are $G(\lambda_i \pi^{-1},  \delta)$ processes, making  $Z^\pi=(Z^{\pi}_t)_{t \geq 0}$ given by
\begin{equation}\label{zpi}
Z^\pi_t=(Z^{\pi,1}_t, \ldots , Z^{\pi,d}_t)
\end{equation}
into a multivariate gamma subordinator with dependent marginals. Therefore, $X^\pi$ enjoys the multivariate subordinated representation \begin{equation}\label{eq:xmultisub}X^\pi=Y_{Z^\pi}.\end{equation}  
In order to further investigate $X^\pi$  we first compute the L\'evy density of $Z^\pi$, which is also of independent interest. Notice that unlike $X^\pi$, $Z^\pi$ is a multivariate process attained by ordinary subordination.

\begin{prop}\label{xpiprop}
The process $Z^\pi$ is a multivariate L\'evy subordinator with zero drift and   L\'evy measure
\begin{equation}\label{eq:rhopi}
\rho^\pi(dt_1 \ldots  dt_d) =\delta \left( \prod_{i=1}  ^d   \frac{e^{-t_i\pi/\lambda_i} - e^{-t_i /\lambda_i}}{t_i}dt_1 \ldots  dt_d   + \sum_{i=1}^d  \frac{e^{-t_i/\lambda_i }}{t_i}d t_i \right)\I_{\{t_i>0\}}.
\end{equation}
\end{prop}

\begin{proof}
By using the p.m.f. \eqref{eq:logdis} 
one can show that the L\'evy measure $r^{\pi}$ of $B^\pi$ is (e.g. \citealt{koz+pod:09})
\begin{equation}\label{eq:NBLevy}
r^{\pi}= \delta \sum_{k =1}^\infty \frac{(1-\pi)^k}{k}\delta_{k},
\end{equation}
where $\delta_k$ is the Dirac measure concentrated in $k$.
With a slight abuse of notation, we write the  multivariate L\'evy density of $Z$ as
\begin{equation}
u_Z(dt_1 \ldots dt_d)=\sum_{i=1}^d \frac{e^{-t_i/\lambda_i }}{t_i} \I_{\{t_i >0\}} d t_i
\end{equation}
and the multivariate independent gamma law $\mu^Z_{t}$ as
\begin{equation}
\mu^Z_{t}=\prod_{i=1}^d f_{Z^i}(t_i; \lambda_i, t ) dt_1 \ldots  dt_d = \prod_{i=1}^d \frac{t_i^{t-1}}{\Gamma(t) \lambda_i^t}e^{-t_i/\lambda_i} \I_{\{t_i>0\}}dt_1 \ldots  dt_d.
\end{equation}
 Using the multivariate version of the (ordinary) subordination integral \eqref{eq:Bochner} (\citealt{sat:99}, Chapter 30) with triplets $(0, 0, u_Z(dt_1 \ldots dt_d))$ and  $(0, \delta, r^{\pi})$, and probability law $\mu^Z_{t}$ we have, applying monotone convergence to interchange integration in $du$ and the series 
\begin{align}\label{eq:rhopiproof}
\rho^\pi(dt_1 \ldots  dt_d)&=\delta  \sum_{k =1}^\infty \frac{(1-\pi)^k}{k} \int_{(0, \infty)} \left( \prod_{i=1}^d \frac{t_i^{u-1}}{\Gamma(u) \lambda_i^u}e^{-t_i/\lambda_i} dt_1 \ldots  dt_d   \right) \delta_{k}(du) +\delta \sum_{i=1}^d  \frac{e^{-t_i/\lambda_i }}{t_i}\I_{\{t_i>0\}} d t_i \nonumber \\ &= \delta  \prod_{i=1}^d   \frac{e^{-t_i/\lambda_i}}{t_i}\I_{\{t_i>0\}}  \sum_{k =1}^\infty \left( \frac{t_i(1-\pi)}{\lambda_i}\right)^k \frac{1}{k!}dt_1 \ldots  dt_d +\delta \sum_{i=1}^d \frac{e^{-t_i/\lambda_i }}{t_i}\I_{\{t_i>0\}}d t_i \nonumber \\ &= \delta   \prod_{i=1}  ^d   \frac{e^{-t_i/\lambda_i }}{t_i}\I_{\{t_i>0\}}  \left(\frac{e^{t_i(1-\pi) /\lambda_i }}{t_i}-1 \right) dt_1 \ldots  dt_d + \delta \sum_{i=1}^d   \frac{e^{-t_i /\lambda_i }}{t_i}\I_{\{t_i>0\}} d t_i ,
\end{align}
which proves  \eqref{eq:rhopi}.
\end{proof}

Together with the foregoing discussion, Proposition \ref{xpiprop} allows the identification of the L\'evy structure of $X^\pi$.




\begin{thm}
\label{MultiTPL}


The process $X^\pi$ is a multidimensional L\'evy subordinator with multivariate characteristic exponent, for \upshape{Re}$(s_i)>0$, $i=1, \ldots, d$ given by
\begin{equation}\label{eq:multiTPLT}
\phi_{X^\pi}(s_1, \ldots, s_d)=\delta \log \left( 1 + \frac{1}{\pi}- \frac{1}{\pi} \prod_{i=1}^d \left(1+ \text{\upshape{sgn}}(\gamma_i) \lambda_i((\theta_i + s_i)^{\gamma_i} -\theta_i^{\gamma_i}  ) \right)\right).
\end{equation} 
Furthermore $X^\pi$ has zero drift, and  L\'evy density
\begin{align}\label{eq:multivarLevy}
&u^\pi_X( x_1,  \ldots, x_d)= \nonumber  \delta \sum_{A \subseteq \{1, \ldots,  d \}}(-1)^{|A|}  \prod_{i \in A} |\gamma_i| \frac{e^{-\theta_i x_i}}{x_i}\left(E_{|\gamma_i|} \left(  \ci  x_i^{ |\gamma_i|}\right) - \I_{\{\text{\upshape{sgn}}(\gamma_i)= -1\}} \right) \nonumber \\ & \phantom{xxxxxxx}  \ \times \prod_{i \in A^c}  |\gamma_i| \frac{e^{-\theta_ix_i}}{x_i}  \left(E_{|\gamma_i|} \left(  \cipi  x_i^{ |\gamma_i|}\right) - \I_{\{\text{\upshape{sgn}}(\gamma_i)= -1\}} \right) \nonumber \\& \phantom{xxxxxxx}  +\delta \sum_{i=1}^d  |\gamma_i| \frac{e^{-\theta_ix_i}}{x_i}  \left(E_{|\gamma_i|} \left(  \ci  x_i^{ |\gamma_i|}\right) - \I_{\{\text{\upshape{sgn}}(\gamma_i)= -1\}} \right)
\end{align}
for $x_i>0$, and zero otherwise, with $(\gamma_i, \lambda_i, \delta_i, \theta_i) \in S$ for all $i=1,\ldots, d$, where the constants $\ci$ and $\cipi$ are given by \eqref{c} with the appropriate parameter modifications.


\end{thm}

\begin{proof}

By definition of $X^\pi$ and independence
\begin{equation}
\phi_{X^\pi}(s_1, \ldots, s_d)=\phi_{B^\pi}(\phi_X(s_1, \ldots , s_d))=\phi_{B^\pi}(\phi_X(s_1) \ldots \phi_X(s_d))
\end{equation}
and \eqref{eq:multiTPLT} is then clear from \eqref{eq:TPLdef} and \eqref{eq:NBlaplace}.

We indicate with $\mu^Y_{t_1, \ldots, t_n}$ the probability law of the random vector $(Y^1_{t_1}, \ldots Y^d_{t_d})$
 and by $\rho^\pi$ the L\'evy measure of $Z^\pi$ given in Proposition \ref{xpiprop}. By virtue of \eqref{eq:xmultisub} we can apply \cite{bn+al:01} Theorem 3.3, and we see that $X^{\pi}$ has L\'evy triplet  $(0,0, \eta^\pi)$ with 

\begin{equation}\label{eq:bochnerMult}
\eta^\pi(B)=\int_{\mathbb R_+^d}\mu^Y_{t_1, \ldots, t_n}(B)\rho^\pi(dt_1 \ldots  dt_n)
\end{equation}
for Borel sets $B \subseteq \mathbb R^d_+$.  
Now by independence we have the product measure
\begin{equation}\label{eq:indepmarg}
\mu^Y_{t_1, \ldots, t_n}=\prod_{i=1}^d \mu^{Y^i}_{t_i}
\end{equation}
where $\mu^{Y^i}_{t}$ indicates the law of $Y_t^i$.
Substituting \eqref{eq:rhopi} and \eqref{eq:indepmarg}  in \eqref{eq:bochnerMult} and using Proposition \ref{LevyTPL} in the second summand of \eqref{eq:rhopi}, we obtain the density for $x_i>0$,
\begin{align}\label{eq:etapi}
u^\pi_X( x_1, \ldots, x_d)
  =\delta \int_{\mathbb R_+^d} \prod_{i=1}^d & f_{Y^i}(x_i;\gamma_i,\theta_i,t_i)\frac{e^{-t_i\pi/\lambda_i} - e^{-t_i /\lambda_i}}{t_i}dt_1 \ldots  dt_d \nonumber \\ &   + \delta \sum_{i=1}^d |\gamma_i| \frac{e^{-\theta_ix_i}}{x_i}   \left( E_{|\gamma_i|} \left( \ci x_i^{|\gamma_i|}\right) \I_{\{\text{\upshape{sgn}}(\gamma_i)= -1\}} \right)
\end{align}
where $f_{Y^i}(x_i;\gamma_i,\theta_i,t_i)$ is given by \eqref{eq:TPSpdf} if $\gamma_i \in (0,1)$, or the absolutely continuous part of \eqref{eq:CPP} if $\gamma_i <0$. Now observe the additive expansion:
\begin{equation}\label{eq:vpifirst}
\prod_{i=1}  ^d   \frac{e^{-t_i\pi/\lambda_i} - e^{-t_i /\lambda_i}}{t_i}=\sum_{A \subseteq \{1, \ldots, d \}}(-1)^{|A|}\exp\left(-\sum_{i\in A}\frac{t_i}{\lambda_i}- \pi \sum_{i\in A^c}\frac{t_i}{\lambda_i}  \right)\prod_{i=1}^d t_i^{-1} 
\end{equation}
 where  $|A|$ denotes the cardinality of the subset $A$. 
 In view of \eqref{eq:vpifirst} we can rewrite the first term in \eqref{eq:etapi} as
  \begin{align}\label{eq:etapi2}
\int_{\mathbb R_+^d}& \prod_{i=1}^d  f_{Y^i}(x_i;\gamma_i,\theta_i,t_i)\frac{e^{-t_i\pi/\lambda_i} - e^{-t_i /\lambda_i}}{t_i}dt_1 \ldots  dt_d  \nonumber \\&= \sum_{A \subseteq \{1, \ldots, d \}}(-1)^{|A|}\left(   \prod_{i \in A} \int_{\mathbb R_+} f_{Y^i}(x_i;\gamma_i,\theta_i,t_i) \frac{e^{- t_i/\lambda_i }}{t_i} dt_i  \prod_{i \in A^c} \int_{\mathbb R_+} f_{Y^i}(x_i;\gamma_i,\theta_i,t_i)\frac{e^{-\pi t_i/\lambda_i }}{t_i} dt_i  \right)
\end{align}
with the product term corresponding to the empty set being one. Replicating  the integrations in Proposition \ref{LevyTPL} we finally arrive at \eqref{eq:multivarLevy}.
\end{proof}

 The L\'evy measure of $X^\pi$ thus decomposes in an independent multivariate TPL measure plus a combinatorial expression of one-dimensional TPL L\'evy measures depending on $\pi$ accounting for the dependence across the marginals, which increasingly gains weight as $\pi$ decreases from one to zero.
In the case $\gamma_i=1$ for all $i$ we notice from \eqref{eq:multiTPLT} that  $X^\pi_t$ follows the multivariate gamma law discussed in \cite{gav:70} and generalizing \cite{kib:41}, which is widely popular for applications.  





\bigskip






\section{Potential for statistical anti-fraud applications}
\label{sec:potential}

One major motivation for our interest in the univariate TPS and TPL laws is their ability to model international trade data, with particular reference to imports into (and exports from) the Member States of the European Union (EU). Due to the combination of economic activities and normative constraints, the empirical distribution of traded quantities and traded values in imports and exports is often markedly skewed with heavy tails, featuring a large number of rounding errors in small-scale transactions due to data registration problems, and structural zeros arising because of confidentiality issues related to national regulations within the EU. While such features are not easy to be combined into a single statistical model, \cite{bar+al:16b} and \cite{bar+al:16a} show that in the univariate case both the TPS and TPL distributions do provide reliable models for the monthly aggregates of import quantities of several products of interest.

The main operational target of the research line on international trade data sketched above is the construction of sound statistical methods for the detection of customs frauds, such as the under-valuation of import duties, and  the investigation of other trade-related infringements, such as money laundering and circumvention of regulatory measures. In this framework flexible statistical models that can accurately describe the distribution of traded quantities and values for a very large number of products is of paramount importance for several reasons. Firstly, such models could provide direct support to policy makers, e.g. in the form of tools for monitoring the effect of policy measures and for providing factual background for the official communications on trade policy. Another goal, which is  perhaps even more prominent from a statistical standpoint, is their use in model-based assessments of the performance of methods used for finding relevant signals of potential fraud.

Most of the fraud detection tools adopted in the context of international trade look for anomalies in the data. Therefore, they typically make use of outlier detection methods for multivariate and regression data, such as those described in \cite{cer:10} and \cite{per+al:20}, as well as of robust clustering techniques \citep[see e.g.][]{CerPer:2014}. All of these techniques assume that the available data have been generated by an appropriate contamination model, which in the context of international trade typically involves at least two variables, in view of the basic economic relationship that yields the value of an individual import (export) transaction as the product of the traded amount and the unit price. Any parameter of the distribution that models the ``genuine'' part of the data must then be estimated in a robust way, in order to avoid the well-known masking and swamping effects due to the anomalies themselves \citep[see e.g.][]{cer+al:19sjs}. Relying on 
the theory of robust high-breakdown estimation, that typically assumes elliptical symmetry of the uncontaminated data-generation process, it is very difficult to derive analytical results for such methods 
when 
skewed distributions should be used for realistic modeling of 
economic processes. The available methods need thus to be compared, evaluated and eventually tuned on a large number of data sets artificially generated with known statistical properties, which must reflect the distributions observed in real-world trade data. \cite{CerPer:2014} show a first attempt in this direction under a rather specialized ad-hoc model. The class of multivariate TPL processes described in Section \ref{sec:multivariate} provides instead a very natural and general reference model, extending the framework suggested by \cite{bar+al:16a} to the simultaneous description of (at least) traded quantities and traded values. Reliable inferential results for anti-fraud diagnostics computed on trade data could then be obtained by simulation from this class of processes following a model-based Monte Carlo scheme, in the spirit e.g. of \cite{bes+dig:77}, \cite{bla+sor:14} and \cite{gue+al:19}.

A similar requirement holds for an alternative approach to fraud detection which has recently attracted considerable attention also in international trade and which rests on the development of powerful and accurate conformance tests of Benford's law \citep{bar+al:18b,cer+al:19,bar+al:21}. This approach aims at unveiling serial fraudsters and has proven to be especially effective for the analysis of individual customs declarations, instead of monthly aggregates of them. A variety of statistical procedures are compared by \cite{cer+al:19} and \citet[Section 7.2]{bar+al:21} through a bootstrap algorithm that generates pseudo-observations mimicking a national database of one calendar-year customs declarations, after appropriate anonymization that makes it impossible to infer the features of individual operators. The class of multivariate TPL processes can again provide a suitable reference framework for such comparisons when a model-based approach replicating international trade conditions is deemed desirable.

Figure \ref{fig:simdata} displays one sample of 5.000 observations from a bivariate TPL process simulated using representation \eqref{xpimulti}. 
The parameters of the marginal processes  are $\gamma_1=\gamma_2=-2.2$, $\lambda_1=\lambda_2=10$, $\theta_1=\theta_2=0.5$, while $\delta=1$ and
$\pi=0.01$ is the success probability relating $X^{\pi,1}$ and $X^{\pi,2}$. The visual similarity between the simulated scatter and the scatter shown in \citet[Section 4, p. 10]{JRC122315} for a homogeneous sample from a fraud-sensitive commodity is striking and confirms the potential of multivariate TPL processes for describing the joint distribution of relevant variables arising in international trade. Therefore, we  argue that suitably tuned versions of $X^\pi$ could lead to reliable simulation inference  for outlier labeling rules and other anti-fraud diagnostics in trade data structures, when the distributional assumption of symmetry for individual uncontaminated observations, typically implied by such methods, is not met. This is an important research goal for anti-fraud applications and international trade analysis, also foreseen in \citet[Section 7.2]{bar+al:21}.

\begin{figure}
\includegraphics[width=0.75\textwidth]{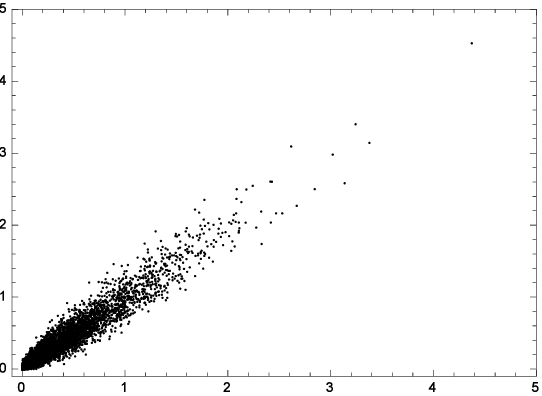}
{\caption[]{\small{Simulated sample from a multivariate TPL process with $d=2$, $\delta=1$, $\pi=0.01$ and marginal parameters given in the text.}} \label{fig:simdata}}
\end{figure}





\section*{Acknowledgments}

This research has financially been supported by the Programme ``FIL-Quota Incentivante'' of  the University of Parma and co-sponsored by Fondazione Cariparma. The authors thank  Peter Carr and Luca Pratelli for the helpful discussions on a previous draft. They are also grateful to   Domenico Perrotta and Francesca Torti of the Joint Research Centre of the European Commission 
for inspiring the anti-fraud applications of the tempered processes described in this work.

\bibliographystyle{apalike}

\bibliography{Bibliography}

\end{document}